\documentclass[25pt]{article}
\usepackage{amsthm,amstext}
\usepackage{amsmath}
\usepackage{graphicx}
\usepackage{amsfonts}
\usepackage{amssymb}
\usepackage{latexsym, xypic}
\usepackage[all]{xy}
\usepackage[colorlinks,linkcolor=blue,citecolor=blue]{hyperref}
\theoremstyle{plain}
\newtheorem{theorem}{Theorem}[section]

\newtheorem{proposition}[theorem]{Proposition}

\theoremstyle{definition}
\newtheorem{definition}[theorem]{Definition}

\theoremstyle{remark}
\newtheorem{remark}{\sc Remark}

\def\pi{positive implicative }

\date{}

\begin{document}
\title{\bf Isomorphism and Fuzzy Subspaces}
\author{\textbf{Iffat Jahan}\\\\       
    Department of Mathematics, Ramjas College\\
    University of Delhi, Delhi, India \\
    ij.umar@yahoo.com \\}
\date{}
\maketitle
    \medskip\noindent
    \begin{abstract}
    This work provides  a necessary and sufficient condition for the isomorphism of two fuzzy subspaces in terms of their dimensions.\\\\
\noindent \text{Keywords:}  Fuzzy subspace; Basis of a fuzzy subspace; Dimension of a fuzzy subspace
    \end{abstract}

    \baselineskip15pt
    \parskip4pt

    \section{Introduction}
The notion of a fuzzy vector space has been introduced by Katsaras and Liu\cite{Katsaras} in 1967.  This notion was further studied and explored in \cite{Abdukhalikov, Lubczonok}. The dimension function of fuzzy subspaces is introduced and discussed by Lubczonok in \cite{ Lubczonok}. Moreover, some basic results on the fuzzy  basis of a fuzzy subspace are also provided therein.  Kumar in \cite{Kumar} has raised an important question regarding the dimension of isomorphic subspaces. In the theory of ordinary vector spaces,  a well known fact is that two vector spaces of finite dimensions are isomorphic if and only if they have the same dimensions. Unfortunately, it is that result, the author \cite{Kumar} wanted to extend to the framework of fuzzy subspaces and discovered that it is not valid. In this work, we provide a complete answer of his question by establishing a necessary and sufficient condition for the isomorphism of two fuzzy subspaces in terms of their dimensions.    \section {Preliminaries}
   Throughout this paper, we shall denote by $U$ and $V$, the ordinary vector spaces over a field $F$. Here we recall same basic definitions and results which are used during the development of present work.
   \begin{definition}\cite{Katsaras}
   	A fuzzy set $\mu$ in $U$ is called a fuzzy subspace of $U$ if for $x,y \in U$ and $\alpha \in F$  \vspace{-.3cm}
   	\begin{enumerate}
   		\item[(i)] $\mu(x-y) \geq min \{\mu(x), \mu(y)\}$,\vspace{-.3cm}
   		\item[(ii)] $\mu(\alpha x) \geq \mu(x)$.
   	\end{enumerate}	 
   The following results on fuzzy subspaces are well known and can be verified easily: \\If $\mu $ is a fuzzy  subspace of $U$ then 
   \begin{enumerate}\vspace{-.3cm}
   	\item [(i)] $\mu (0) \geq \mu (x) $ for $x\in U$,\vspace{-.3cm}
   	 \item [(ii)] $\mu(\alpha x) = \mu(x)$ for $x\in U$ and $\alpha \in F$.
   \end{enumerate} 
   \end{definition}
   \noindent Note that  $sup~\mu= \displaystyle {sup_{x \in U} \{\mu(x)\}}=\mu(0).$
\noindent Now onwards, $\mu $ and $Im\mu$  shall denote a fuzzy  subspace of $U$ and the image set $\{\mu(x):x \in U\}$ respectively. Further, we recall:
\begin{theorem} \cite{Kumar} \label{2.2}
Let $\mu$ be a fuzzy subset of $U$. Then, $\mu$ is a fuzzy subspace of $U$ if and only if each non-empty level subset $\mu_t,~t \in [0,1]$, is a subspace of $U$.
\end{theorem}
\noindent Unless mentioned otherwise, we shall use the terminology of level subspaces for level subsets $\mu_t$, where $t \in [0,1]$.

\noindent The following results and definitions are due to Lubczonok  : 
\begin{definition}\cite{Lubczonok}
	Let $\mu$ be a fuzzy subspace of $U$. Then, an ordinary basis $\beta$ of $U$ is called a fuzzy basis of $\mu$, if
	\[\mu ( \alpha_1 x_1+\cdots+\alpha_n x_n)= min \{\mu(x_1), \cdots \mu(x_n)\}, \]
\end{definition} \vspace{-.3cm}
\noindent when $\{x_1, \cdots, x_n\}$ is any finite subset of $\beta$ and $\alpha_{i} \in F$.

\begin{theorem} \cite{Lubczonok} \label{2.6}
Every fuzzy subspace of a finite dimensional vector space has a fuzzy basis.
\end{theorem}
\remark Let $\mu$ be a fuzzy subspace of $U$ such that $dim \mu =n$. Then, $|Im~\mu| \leq n+1$ where $|Im~\mu| $ denotes the cardinality of $Im~ \mu$.
\begin{theorem} \cite{Lubczonok}
Let $\mu$ be a fuzzy subspace of a finite dimensional vector space $U$. If $\beta$ and $\beta^*$ are any two fuzzy basis of $\mu$, then 
\[\sum_{x \in \beta} \mu(x) = \sum_{x \in \beta^*} \mu(x).\]
\end{theorem}
\noindent The above theorem is a prerequisite for the following definition of $dim~\mu$:
\begin{definition}\cite{Lubczonok}
Let $\mu$ be a fuzzy subspace of a finite dimensional vector space $U$. The \textbf{dimension} of $\mu$, denoted by $dim~\mu$ is defined as
\[dim~ \mu=\sum_{x \in \beta } \mu (x).\]

\end{definition}
\begin{definition}\cite{Zadeh}
	Let $f$ be any mapping from a set $X$ to a set $Y$ and $\mu$ be a fuzzy set in $X$. Then the image $f(\mu)$ of $\mu$ under $f$ is defined by \\
$$f(\mu)(y)=  \sup_{x\in f^{-1}(y)}\{\mu(x)\}~for~all~y\in Y.$$
\end{definition}
\noindent The following result is obvious:
\begin{proposition}\cite{Kumar}
Let $f$ be a homomorphism from a vector space $U$ to a vector space $V$. If $\mu$ is a fuzzy subspace of $U$, then $f(\mu)$ is a fuzzy subspace of $V$.
\end{proposition}
\begin{definition} \cite{Kumar}Let $f$ be a homomorphism from a vector space $U$ to a vector space $V$. If $\mu$ is a fuzzy subspace of $U$, then $f(\mu)$ is called the homomorphic image of $\mu$ under $f$. In addition, if $f$ is an isomorphism from $U$ onto $V$, the $\mu$ and $f(\mu)$ are called isomorphic to each other.
\end{definition}
\begin{definition}\cite{Kumar}
Let $U$ and $V$ be vector spaces. Let $\mu$ and $\eta$ be fuzzy subspaces of $U$ and $V$ respectively. Then, $\eta$ is said to be isomorphic to $\mu$ if exists an isomorphism $f$ from $U$ onto $V$ such that $f(\mu) = \eta$. We denote it by $\mu \approx \eta$. It is easy to see that $\eta(0)=f(\mu)(0)=\mu(0)$ and   $Im\mu= Im\eta $.
\end{definition}
\noindent Recall the following well known results from the literature:
\begin{proposition} \label{2.13_1}
Let  $f:X \rightarrow Y$ be a one-one and onto map. Then, for a fuzzy subset $\mu$ in $X$ and  $t\in[0,1]$,  $f(\mu)_t=f(\mu_t)$. 
\end{proposition}

\begin{proposition} \label{2.14_1}
Let   $\mu$ and $\eta$ be fuzzy subsets in  $X$. Then,
  \begin{enumerate} \item [(i)]$\mu=\eta$ if and only if $\mu_t=\eta_t$ for all $t \in [0,1]$,
\item[(ii)] $\mu_a=\mu_t \subsetneq \mu_s$ where  $t$ and $s \in Im\mu$ such that $s<a \leq t$ and there is no element of $Im \mu$ between $s$ and $t$.
 \end{enumerate}
\end{proposition}
\noindent  Here we mention that the Theorem 3.4 in \cite{Kumar} is not presented correctly. Below, we reformulate  this theorem and provide its proof:
\begin{theorem}
Let $\mu$ and $\theta$ be fuzzy subspaces of $U$ and $V$ respectively. Then, 
 \begin{center}
 $\mu\approx \eta $ if and only if $f(\mu_t)= \eta_t$ where $f:U \rightarrow V$ is an onto isomorphism and $t\in [0,1]$.  
 \end{center}
\end{theorem}
\begin{proof}
\textbf{Necessity.} Let $\mu\approx \eta $. Then, there  exists an isomorphism $f$ from $U$ onto $V$ such that $f(\mu) = \eta$. Then, for $t\in [0,1]$, in view of Proposition \ref{2.13_1}, we have $ \eta_t = f(\mu)_t= f(\mu_t)$.\\
\textbf{Sufficiency.} Let $f:U \rightarrow V$ be an onto isomorphism and $f(\mu_t)= \eta_t$ for all $t \in [0,1]$. Then, by Proposition \ref{2.13_1} and Proposition \ref{2.14_1}, we have $f(\mu)= \eta$. Hence $\mu\approx \eta $.
\end{proof}

\noindent Moreover, the following result is immediate :
\begin{proposition} \label{2.13}Let $\mu$ be a fuzzy subspace of a vector space $U$ and $\beta $ be a fuzzy basis of $\mu$. Also, let $A$ be an ordinary subspace. Then, \vspace{-.3 cm}
\begin{enumerate}
\item[(i)] the restriction $\mu \left | A\right.$ is a fuzzy subspace of $A$,\vspace{-.3 cm}
\item[(ii)] the subset $\beta \cap A$ is a fuzzy basis of $\mu \left | A\right.$. 
\end{enumerate} 
\end{proposition}
\noindent Moreover, we have:
\begin{theorem} \cite{Lubczonok}
Let $\mu$ be a fuzzy subspace of a vector space $U$ such that $Im\mu$ is upper well ordered. Let $A$ be a proper subspace of $U$ and $\beta_A$ be a fuzzy basis of $\mu|A$. Then, there exists $w\in U\sim A$ such that $\beta^*=\beta \cup\{w\}$ is a fuzzy basis of $\mu|A^*$ where $A^*=\langle A\cup \{w\} \rangle$.
\end{theorem}
\noindent Finally, by using the above theorem repeatedly, the following can be proved easily:
\begin{theorem}\cite{Lubczonok} \label{2.5}
Let $\mu$ be a fuzzy subspace of a finite dimensional vector space $U$. Let $A$ be a proper subspace of $U$. Let $\beta_A = \{x_1, x_2, \cdots, x_k\}$ be a fuzzy basis of $\mu \left | A\right.$. Then, there exist a fuzzy basis $\beta=\{x_1, x_2, \cdots, x_k, x_{k+1}, \cdots, x_n\}$ of $\mu$. 
\end{theorem}

\noindent For the sake of completeness, below we prove the result from \cite{Kumar} related to isomorphic image of a fuzzy basis.
\begin{theorem} \label{2.14}
Let $f$ be an isomorphism from a vector space $U$ onto a vector space $V$ and $\mu$ be a fuzzy subspace of $U$. Let $U$ and $V$ be~~ finite dimensional~~~ vector ~~spaces.  If \mbox{$\beta = \{x_1, x_2, \cdots, x_n\}$} is a fuzzy basis of $U$, then $f(\beta)$ is a fuzzy basis of the fuzzy subspace $f(\mu)$ where \mbox{$f(\beta)=\{f(x_1),f(x_2), \cdots, f(x_n)\}$.}
\end{theorem}
\begin{proof}
Let $\beta = \{x_1, x_2, \cdots, x_n\}$ be a fuzzy basis of $U$. Then, $\beta$ is a basis of $U$. It is well known that for an isomorphism $f$ from $U$ onto $V$, $f(\beta)=\{f(x_1),f(x_2), \cdots, f(x_n)\}$ forms a basis of $f(U)=V$. Now, to show that $f(\beta)$ is a fuzzy basis of $f(U)=V$, we only need to establish that $f(\beta)$ is fuzzy linearly independent. That is if $c_1, c_2, \cdots, c_n$ are scalars, then we shall show
$$f(\mu)(c_1f(x_1)+\cdots+c_nf(x_n) = min~\{f(\mu)\left (f(x_1)\right), \cdots, f(\mu)\left (f(x_n)\right)\}. $$ 
So, for the scalers $c_1, c_2, \cdots, c_n$, consider
\begin{eqnarray*}
f(\mu)(c_1f(x_1)+\cdots+c_nf(x_n)) &=&
f(\mu)\left(f(c_1 x_1 +\cdots+c_n x_n) \right)\\
&=& \mu (c_1 x_1 +\cdots+c_n x_n) ~~~~\text(as~ f ~is~ a~one~	to~one~map)\\
&=& min~\{\mu ( x_1), \cdots, \mu(x_n)\}\\&&~~~~~~~~~~~~~~~~\text(as~ \beta ~is~ fuzzy~linearly~	independent)\\
&=&min~\{f(\mu) (f( x_1)), \cdots, f(\mu)(f(x_n))\}.
\end{eqnarray*}
This completes the proof.
\end{proof}

  \section{Isomorphism and Fuzzy Subspaces}
For certain convenience, we follow a convention. For each $t\in [0,1]$, we have a level subset $\mu_t$ in $U$, and  consider the restriction $\mu | _{\mu_t}$ for $t\in [0,1]$. Since $U$ is finite dimensional, $\mu_t$ is also a finite dimensional subspace for each $t\in [0,\mu(0)]$. Therefore we can talk of dimension function of fuzzy subspace $\mu|_{\mu_t}$. On the other hand, for $t \in~ ]\mu(0), 1]$, the level subset $\mu_t $ is an empty set, and we set $dim~\mu| _{\mu_t} =-1$.

\begin{theorem}
	Let $\mu $ and $\eta$ be fuzzy subspaces of finite dimensional vector spaces $U$ and $V$, respectively. Then,
	$$
	\mu \approx \eta \text{~~if and only if~~} dim ( {\mu}\left |_{\mu_t} \right.)=dim ( {\eta}\left |_{\eta_t} \right.)~\text{for}~t\in[0,1].
	$$
\end{theorem}
\begin{proof}
\textbf{	Necessity.} Since $\mu \approx \eta $, there exists an isomorphism $f$ from $U$ onto $V$, such that $f(\mu)=\eta$. Now, let $t \in [0,1]$. It is easy to see that $\eta(0)=f(\mu)(0)=\mu(0)$. Thus if $t\in]\mu(0), 1]$, then clearly, $dim ( {\mu}\left |_{\mu_t} \right.)=-1=dim ( {\eta}\left |_{\eta_t} \right.)$. So we let, $t\in[0,\mu(0)]$. Then by Theorem \ref{2.2}, $\mu_t$ is a subspace of $U$. Thus, in view of Theorem \ref{2.6}, there exists a fuzzy basis of the fuzzy subspace 
${\mu}\left |_{\mu_t} \right.$ of $\mu_t$. Let 
$\beta_t=\{x_1,x_2\cdots, x_k\}$ be a fuzzy basis of ${\mu}\left|_{\mu_t} \right.$. Observe that $f|\mu_t $ is an isomorphism from $\mu_t$ to $f(\mu_t)$. So by Theorem \ref{2.14}, $f(\beta_t)$ is a fuzzy basis of $f(\mu)|f(\mu_t)$. Moreover, by Theorem \ref{2.5}, $\beta_t$ can be extended to form a fuzzy basis $\beta=\{x_1,x_2\cdots, x_k, x_{k+1}\cdots x_n\}$ of $\mu$. Again as , $f$ is an isomorphism, by Theorem \ref{2.14}, $f(\beta)$ is a fuzzy basis of $f(\mu)$. Since $f$ is an isomorphism, $f(\mu)=\eta$ and 
$f(\mu_t)=f(\mu)_t =\eta_t$ (by Proposition \ref{2.13_1}). Thus,  
$${f(\mu)}\left|_{f(\mu_t)} \right.={\eta}\left|_{\eta_t} \right.$$ 
This implies that
$$
dim~{\eta}\left|_{\eta_t} \right.=dim~ {f(\mu)}\left|_{f(\mu_t)} \right. = \sum_{f(x)\in f(\beta_t)}f(\mu)|_{f(\mu_t)}(f(x))=\sum_{x \in\beta_t} \mu(x)=dim ~\mu \left| \right.\mu_t.
$$
\textbf{Sufficiency.} Firstly we claim that 
$$
\inf_{t_i\in [0,1]} \{t_i : \mu_{t_i} = \phi \}=  \inf_{t_i\in [0,1]} \{t_i : \eta_{t_i}= \phi \}.
$$
Suppose if possible 
$$
\inf_{t_i\in [0,1]} \{t_i : \mu_{t_i} = \phi \} <  \inf_{t_i\in [0,1]} \{t_i : \eta_{t_i}= \phi \}.
$$
Then there exists $t_r \in[0,1]$ with $\mu_{t_r} = \phi $ and 
$$
t_r  <  \inf_{t_i\in [0,1]} \{t_i : \eta_{t_i}= \phi \}.
$$
This implies that $\eta_{t_r} \neq \phi$ and hence  $dim~\left( \eta|\eta_{t_r}\right) \neq -1$. But as
$\mu_{t_r} = \phi $,  \mbox{$dim~\left( \mu|\mu_{t_r}\right)=-1$}. Thus 
$$
dim ( {\mu}\left |_{\mu_{t_r}} \right.) \neq dim ( {\eta}\left |_{\eta_{t_r}} \right.)~\text{for}~t_r\in[0,1].
$$
This contradiction establishes our claim. Let us write  
$$
t_{0}=\inf_{t_i\in [0,1]} \{t_i : \mu_{t_i} = \phi \} =  \inf_{t_i\in [0,1]} \{t_i : \eta_{t_i}= \phi \}.
$$
Here we observe that as $sup~ \mu =\mu(0)$ and $sup~ \eta =\eta(0)$, we have $\mu_{t_i} = \emptyset $ for all $t_i > \mu(0)$ and $\eta_{t_j} = \emptyset $ for all $t_j > \eta(0)$. Thus, $inf\{t_i :\mu_{t_i} = \emptyset\} = \mu(0)$ and $inf\{t_j :\eta_{t_j} = \emptyset \}= \eta(0)$. Consequently,
$$\mu (0) =t_0=\eta(0).$$ 
Now, in view of the hypothesis  
\begin{eqnarray*}
dim ( {\mu}\left |_{\mu_{t_0}} \right.) = dim ( {\eta}\left |_{\eta_{t_0}} \right.),
\end{eqnarray*}
where $\mu_{t_0} = \{x \in U : \mu(x)= t_0\}$ and $\eta_{t_0} = \{x \in U : \eta(x)= t_0\}$.
Then, $\mu_{t_0}$ and $\eta_{t_{0}}$ are ordinary subspaces of $U$ and $V$ respectively. Our next claim is that \begin{center}$\mu_{t_0}$ and $\eta_{t_{0}}$ have same dimensions.\end{center} Let $\beta (\mu_{t_0})$ and $\beta (\eta_{t_0})$ be the basis of $\mu_{t_0}$ and $\eta_{t_0}$ respectively.
Let $dim~ \mu_{t_0} = k_1$  and \mbox{$dim~ \eta_{t_0} = k_2$}. Then, 
$$
dim ( {\mu}\left |_{\mu_{t_0}} \right.) =  \sum_{x \in \beta (\mu_{t_0})}  {\mu}\left |_{\mu_{t_0}} \right. (x) = k_1t_0.
$$
Also, 
$$
dim ( {\eta}\left |_{\eta_{t_0}} \right.) =  \sum_{x \in \beta (\eta_{t_0})}  {\eta}\left |_{\eta_{t_0}} \right. (x) = k_2t_0.
$$
\noindent Thus in view of the hypothesis, $ k_1t_0 = k_2t_0$. Hence, $k_1 = k_2$. This proves the claim. \\
Next, consider the sets $Im~\mu \sim \{t_0\}$ and  $Im~\eta \sim \{t_0\}$. Then, we claim that
$$
sup\{ t_{i} : t_i \in Im ~\mu \sim \{t_0\}\} = sup\{ t_{i} : t_i \in Im ~\eta \sim \{t_0\}\}.
$$
Suppose, if possible 
$$
sup\{ t_{i} : t_i \in Im ~\mu \sim \{t_0\}\} > sup\{ t_{i} : t_i \in Im ~\eta \sim \{t_0\}\}.
$$
Then for some $t_s \in Im \mu \sim \{t_0\}$, we have 
$$
 t_s > sup\{ t_{i} : t_i \in Im ~\eta \sim \{t_0\}\}.
$$
Also note that $t_0 > t_s$. Since $t_s \notin  Im \eta$, and $t_0 \in Im~\eta$, by Proposition \ref{2.14_1}, we have  $\eta_{_{t_0}} = \eta_{_{t_s}}$. Therefore,
$$
dim \left( \eta | \eta_{_{t_s}}\right) = dim \left( \eta | \eta_{_{t_0}}\right)=k_1t_0,
$$
where $dim~ \eta_{t_0} = k_1$.
Let $\beta (\mu_{t_s})$  be a  basis of $\mu | \mu_{t_s}$. Since $\mu_{_{t_0}}$  is a proper subspace of $\mu_{_{t_s}}$ the dimension of $\mu_{_{t_s}}$ is strictly greater than that of $\mu_{_{t_0}}$. Thus if $\beta_(\mu_{t_0})$ is a basis of $\mu_{t_0}$ and $\beta_(\mu_{t_s})$  is an extended basis for $\mu_{t_s}$ obtained from $\beta_(\mu_{t_0})$, we have 
 $$k_1= dim ~\mu_{t_0}< dim~ \mu_{t_s}.$$
Consequently, 
$$
dim \left ( \mu | \mu_{t_s} \right.) =  \sum_{x \in \beta (\mu_{t_s})}  {\mu}\left |_{\mu_{t_s}} \right. (x)  > k_1t_0.
$$
Thus 
$$
dim \left ( \mu | \mu_{t_s} \right.) \neq dim \left ( \eta | \eta_{t_s} \right.)~\text{where}~t_s \in Im ~\mu \sim \{t_0\}.
$$
This contradiction establishes our claim. Now, we write 
$$
t_1=sup\{ t_{i} : t_i \in Im ~\mu \sim \{t_0\}\} = sup\{ t_{i} : t_i \in Im ~\eta \sim \{t_0\}\}.
$$
By repeating the above process for the sets 
$Im ~\mu \sim \{t_0, t_1\}$ and $Im ~\eta \sim \{t_0, t_1\}$, we obtain $t_2 \in [0, 1]$ such that 
$$
t_2= sup\{ t_{i} : t_i \in Im ~\mu \sim \{t_0, t_1\}\} = sup\{ t_{i} : t_i \in Im ~\eta \sim \{t_0, t_1\}\}.$$ 
Here we mention that as both $U$ and $V$ are finite dimensional vector spaces, both $Im \mu$ and $Im \eta$ are finite(by Remark 1). Thus  there is a decreasing sequence $t_0, t_1, \cdots t_n$ in $Im~\mu$ and $Im~\eta$ such that 
$$ Im~\mu=\{t_0, t_1, \cdots,  t_n\}=Im~\eta.$$
Note that $\mu_{_{t_n}}= U$ and $\eta_{_{t_n}}= V$. Now consider the following chain of subspaces 
$$\mu_{_{t_0}} \subseteq \mu_{_{t_1}}\subseteq \cdots \subseteq \mu_{_{t_n}}=U,$$
and
$$\mu_{_{t_0}} \subseteq \mu_{_{t_1}}\subseteq \cdots \subseteq \mu_{_{t_n}}=V.$$
Let $\beta(\mu_{_{t_0}})$ be a basis of $\mu_{_{t_0}}$. We construct a basis of $\mu_{_{t_1}}$ as follows:
Choose an element $e_1 \in \mu_{_{t_1}}\sim \mu_{_{t_0}}$ and generate a subspace $U_{(0,1)}$  with a basis $\beta(\mu_{_{t_0}}) \cup \{e_1\}$. Again, choose an element $e_2 \in \mu_{_{t_1}}\sim U_{(0,1)}$ and consider the subspace $U_{(0,2)}$ generated by $\beta(\mu_{_{t_0}}) \cup \{e_1, e_2\}$. Repeating this process, we obtain subspaces $U_{(0,1)}, U_{(0,2)} \cdots U_{(0,r_0)}= \mu_{_{t_1}}$ such that 
$$U_{(0,1)}\subseteq U_{(0,2)} \subseteq\cdots \subseteq U_{(0,r_0)}= \mu_{_{t_1}}$$ 
The basis of $\mu_{_{t_1}}$, so obtained is denoted by $\beta(\mu_{_{t_1}})$ and $\beta(\mu_{_{t_1}})=\beta(\mu_{_{t_0}}) \cup \{e_1,e_1 \cdots e_{_{r_0}} \}.$
We construct a basis of $\mu_{_{t_2}}$ by following the above procedure. During the process we obtain subspaces $U_{(1,1)}, U_{(1,2)} \cdots U_{(1,r_1)}= \mu_{_{t_1}}$ such that 
$$U_{(1,1)} \subseteq  U_{(1,2)} \subseteq \cdots\subseteq U_{(1,r_1)}= \mu_{_{t_2}}.$$ 
That is, by the above procedure, we obtain subspaces  $U_{(i,1)}, U_{(i,2)} \cdots U_{(i,r_i)}= \mu_{_{t_{i+1}}}$ for each pair of level subsets $\mu_{_{t_{i}}}$ and $\mu_{_{t_{i+1}}}$ such that
$$U_{(i,1)} \subseteq  U_{(i,2)} \subseteq \cdots\subseteq U_{(i,r_i)}= \mu_{_{t_{i+1}}}.$$
Let $\beta(\mu_{_{t_{i}}})$ be the basis of $\mu_{_{t_{i}}}$ obtained in this process. The costruction of $\beta(\mu_{_{t_{i}}})$  ensures that
 $$ \beta(\mu_{_{t_{0}}}) \subseteq \beta(\mu_{_{t_{1}}}) \subseteq \cdots \subseteq \beta(\mu_{_{t_{n}}}).$$
Similarly we can construct a chain of basis of subspaces $\eta_{_{t_{i}}}$ starting with a basis 
$\beta(\eta_{_{t_{0}}})$  of $\eta_{_{t_{0}}}$. Let the chain of such basis be 
$$ \beta(\eta_{_{t_{0}}}) \subseteq \beta(\eta_{_{t_{1}}}) \subseteq \cdots \subseteq \beta(\eta_{_{t_{n}}}).$$
Clearly,  $|\beta (\mu_{_{t_i}})|=|\beta (\eta_{_{t_i}})|$ and $\beta (\mu_{_{t_n}})$ and $\beta (\eta_{_{t_n}})$ are basis of $U$ and $V$ respectively. Now, let $f$ be any mapping which maps $\beta (\mu_{_{t_0}})$ onto $\beta (\eta_{_{t_n}})$ and satisfies 
$$f(e_i)=e_{i}^{'} \text {~for~} e_i\in \beta (\mu_{_{t_i}}) \sim \beta (\mu_{_{t_{i-1}}}),$$
where $e_{i}^{'}\in \beta (\eta_{_{t_i}}) \sim \beta (\eta_{_{t_{i-1}}})$. Extend $f$ by linearity from $U$ to $V$. Then, $f$ is an isomorphism from $U$ to $V $ such that $f(\mu)=\eta$. This completes the proof of the theorem.
\end{proof}
\section*{Acknowledgement}
I dedicate this work to Dr. Naseem Ajmal, one of the greatest fuzzy algebraist and my Ph. D. supervisor.

    \end{document}